\documentclass[10pt]{amsart}
\usepackage{amssymb,amsmath,amsfonts,amsthm}
\usepackage[dvips]{graphicx}
\usepackage{psfrag}

\theoremstyle{plain}
\newtheorem{main}{Theorem}

\newtheorem{theorem}{Theorem}[section]
\newtheorem{lemma}[theorem]{Lemma}
\newtheorem{proposition}[theorem]{Proposition}
\newtheorem{corollary}[theorem]{Corollary}
\theoremstyle{remark}
\newtheorem{remark}[theorem]{Remark}

\newcommand{\Leb}{\operatorname{vol}}

           \def\ea{\end{array}}
          \def\ec{\end{center}}
     \def\ed{\end{description}}
        \def\ee{\end{equation}}
       \def\eea{\end{eqnarray}}
     \def\eeaa{\end{eqnarray*}}
 \def\et{\end{thebibliography}}

\def\Diff{{\rm Diff}}

\def\Gibb{{\rm Gibb}}
\def\inte{{\rm int}}

\def\supp{\operatorname{supp}}

\def\cC{{\mathcal C}}

\def\cK{{\mathcal K}}
\def\cI{{\mathcal I}}

\def\cU{{\mathcal U}}

\def\cB{{\mathcal B}}
\def\cH{{\mathcal H}}

\def\cF{{\mathcal F}}

\def\cP{{\mathcal P}}

\title{Partially volume expanding diffeomorphisms}
\author{Shaobo Gan, Ming Li, Marcelo Viana and Jiagang Yang}
\date{\today}
\thanks{S.G. is supported by NSFC 11231001 and NSFC 11771025. M.L. is supported by NSFC 11571188. M.V. and J.Y. were partially supported by CNPq, FAPERJ, and PRONEX.
This work was supported by the Fondation Louis D-Institut de France (project coordinated by M. Viana)}

\address{School of Mathematical Sciences, Peking University, Beijing 100871, China}
\email{gansb\@@pku.edu.cn}

\address{School of Mathematical Sciences and LPMC, Nankai University, Tianjin 300071, People's Republic of China}
\email{limingmath\@@nankai.edu.cn}

\address{IMPA, Est. D. Castorina 110, 22460-320 Rio de Janeiro, Brazil}
\email{viana\@@impa.br}

\address{Departamento de Geometria, Instituto de Matem\'atica e Estat\'\i stica, Universidade Federal Fluminense, Niter\'oi, Brazil}
\email{yangjg\@@impa.br}

\begin{document}

\begin{abstract}
We call a partially hyperbolic diffeomorphism \emph{partially volume expanding} if the Jacobian restricted to any hyperplane
that contains the unstable bundle $E^u$ is larger than $1$. This is a $C^1$ open property.
We show that any $C^{1+}$ partially volume expanding diffeomorphisms admits finitely many physical measures, the union of whose
basins has full volume. %We also provide a class of solenoid-type attracting sets that are partially volume expanding.
\end{abstract}

\maketitle

\tableofcontents

\setcounter{tocdepth}{1} \tableofcontents

\section{Introduction}
Let $f:M\to M$ be a diffeomorphism on some compact Riemannian manifold $M$.
An invariant probability $\mu$ of $f$ is a \emph{physical measure} if the set of points $z\in M$ for which
\begin{equation}\label{eq.basin}
\frac{1}{n}\sum_{i=0}^{n-1}\delta_{f^i(x)}\to \mu \text{(in the weak$^*$
sense)}
\end{equation}
has positive volume. This set is denoted by $\cB(\mu)$ and called the \emph{basin} of $\mu$.

In the present paper, we investigate the existence of existence and finiteness of physical measures in the
setting of partially hyperbolic diffeomorphisms. More precisely, we assume that there exists an
splitting $TM = E^u \oplus E^{cs}$ of the tangent bundle that is invariant under the tangent map $Df$ and
satisfies
\begin{equation}\label{eq.dfpartially}
\|(Df\mid_{E^{u}_x})^{-1})\| < 1 \text{ and } \|(Df\mid_{E^{u}_x})^{-1})\| \|Df \mid_{E^{cs}_x}\| < 1  \text{ at every } x \in M.
\end{equation}
In other words, the \emph{unstable bundle} $E^u$ is uniformly expanding, and it \emph{dominates}
the \emph{center-stable bundle} $E^{cs}$.

A program
for investigating the physical measures of partially hyperbolic diffeomorphisms was initiated by Alves, Bonatti, Viana
in \cite{BoV00,ABV00}. Their starting point was the observation that physical measures must be Gibbs $u$-states,
a notion they borrowed from Pesin, Sinai~\cite{PS82}. Not all Gibbs $u$-states are physical measures, but that is
easily seen to be the case for those Gibbs $u$-states whose center stable Lyapunov exponents are all negative.
Bonatti, Viana~\cite{BoV00} introduced the notion of \emph{mostly contracting center}, and proved that under this condition
there are finitely many ergodic Gibbs $u$-states, all with negative center Lyapunov exponents, and they are the
physical measures. Moreover, the union of their basins is a full volume set.

Further results on physical measures of diffeomorphims with mostly contracting center have been
obtained by Dolgopyat~\cite{Dol00}, Castro~\cite{Cas02}, Burns, Dolgopyat, Pesin~\cite{BDP02},
De Simoi, Liverani~\cite{DSL16}, Dolgopyat, Viana, Yang~\cite{DVY16} and others.

Here we take a different, although related viewpoint. One key new observation  (Theorem~\ref{t.limit} below) is that,
for any $C^1$ diffeomorphism, every physical measure $\mu$ must be \emph{volume non-expanding} in the sense that
$$
\int_M \log|\det Df| \, d\mu \le 0.
$$
Thus, going back to the partially hyperbolic setting, we may restrict our analysis to volume non-expanding Gibbs
$u$-states. We say that the diffeomorphism $f$ is \emph{partially volume expanding} if
$$
\left|\det Df(x) \mid_{H}\right| > 1
$$
for any codimension-one subspace $H$ of $T_xM$ that contains $E_x^u$.

Being partially volume expanding is clearly a $C^1$ open property (the corresponding statement for mostly contracting
cente is more subtle, and was proven by Andersson~\cite{An10} and Yang~\cite{Yang-partial}).
Moreover, it is not difficult to find examples. For instance, let $S^1$ be the circle and $D$ be the $2$-dimensional disk.
Then consider an embedding $f_0: M \to M$ of the solid torus $M=S^1\times D$ of the form
\begin{equation}\label{eq_solenoid}
f_0(\theta,x)=\left(k\theta \mod 1, h_\theta(x)\right)
\end{equation}
where $k \ge 3$ is an integer larger and $h_\theta(x)$ is such that $\|Dh_\theta(x)\|$ and $\|Dh_\theta(x)^{-1}\|$
are both strictly less than $k$ at every point.
As we will check later, the first condition implies that $f_0$ is partially hyperbolic, and the second one ensures
that it is partially volume expanding.

\begin{main}\label{main.global}
Any partially volume expanding $C^{1+}$ diffeomorphism admits finitely many physical measures, the union of whose basins is
a full volume subset of the ambient manifold.
\end{main}

Several other results on the physical measures of partially hyperbolic maps have been obtained, especially in the setting
of maps \emph{mostly expanding cente} that was introduced by Alves, Bonatii, Viana~\cite{ABV00}.
Andersson, V\'asquez~\cite{AnV18} use a slightly stronger definition, which they prove is $C^2$ open.
The latter was improved by Yang~\cite{Yang-partial}, who proved $C^1$ openness. Bifurcation properties of the
physical measures have been studied by Andersson, V\'asquez~\cite{AnV} and Yang~\cite{Yang-expanding}

The main novelty of our present results, with respect to those more studied cases, is perhaps that the
notion of partial volume expansion requires no assumptions on the signs of the Lyapunov exponents:
it allows us to focus on the more significant Gibbs $u$-states, with negative center exponents,
and disregard all the other ones.

This comes with a price: while the conclusion of Theorem~\ref{t.limit}
remains true for nearby maps, just because the assumptions are $C^1$ open, we have no control on how the
number of physical measures unfolds under perturbation. This is in contrast with the case of diffeomorphisms
with mostly contracting center, where a very precise bifurcation theory for physical measures exists
(see \cite{DVY16,HYY,Yang-partial}) that describes the number and supports of physical measures for all the
perturbations of an initial map.

We call \emph{$u$-codimension} of a partially hyperbolic diffeomorphism the dimension of its center sub-bundle.
It is clear from the definition that any partially hyperbolic diffeomorphism with $u$-codimension 1 is
partially volume expanding. The theorem that follows provides a new way to establish the property of mostly
contracting center. In particular, it implies that the generalized solenoid $f_0:M \to M$ presented above
has mostly contracting center, which was not known previously.

\begin{main}\label{main.relation}
Every partially volume expanding $C^{1+}$ diffeomorphism that has $u$-codimension less than or equal to 2
has mostly contracting center.
\end{main}

In particular, any partially hyperbolic diffeomorphism with $u$-codimension 1 has mostly contracting center.
This does not seem to have been pointed out before.
It is easy to find examples (e.g. Smale solenoids) of $u$-codimension 2 partially hyperbolic diffeomorphisms
that have mostly contracting center, and yet fail to be partially volume expanding.
Although we ere not able to produce a counterexample, we believe that Theorem~\ref{main.relation}
does not extend to $u$-codimension 3.

In Section~\ref{local} we extend these ideas to a semi-local setting, namely to
partially hyperbolic attracting sets of embeddings $M\to\inte(M)$ of compact
manifolds with boundary, including the generalized solenoid $f_0:M \to M$ in
\eqref{eq_solenoid}, and its perturbations. Note that our conditions on $f_0$
are a lot more flexible than in the usual construction of the Smale
solenoid~\cite{Shu87,Sma67}: in particular, they include many non-hyperbolic examples. In fact, it is clear that  the same ideas can be applied to more general solenoids,
including the natural extensions of expanding maps on (possibly branched)
manifolds studied by Williams~\cite{Wil74}, Bonatti, Pumari\~{n}o, Viana~\cite{BPV00}
and Du~\cite{Du13}. We are grateful to Bin Yu for pointing this out to us.

\section{Preliminaries}

In this section, we collect a number of classical notations and facts that are useful for our arguments.
We start by reinterpreting the partial volume expansion assumption.

\subsection{Partial volume expansion}

Let $f:M\to M$ be a $C^1$ diffeomorphism and $\mu$ be an invariant probability measure.
For $\mu$-almost every $x$, let $l=l(x)\ge 1$ and $\lambda_1(x) > \cdots > \lambda_l(x)$ be the Lyapunov exponents and
$$
T_xM = E_x^1 \oplus \cdots \oplus E_x^l
$$
be the Oseledets splitting, given by the Oseledets multiplicative ergodic theorem (Oseledets~\cite{Ose68},
see \cite[Theorem~4.2]{LLE}).

Denote $\Delta(x) = \sum_{i=1}^l \lambda_i(x) \dim E_x^i$. The Oseledets theorem also gives that
\begin{equation}\label{eq.Delta}
\lim_n \frac 1n \log \left|\det Df^n(x)\right| = \Delta(x)
\end{equation}
for $\mu$-almost every $x$. Similarly, for any hyperplane $H \subset T_x M$ there exists $i$ such that
\begin{equation}\label{eq.used1}
\lim_n \frac 1n \log \left|\det Df^n(x) _{\mid H}\right| = \Delta(x)-\lambda_i(x).
\end{equation}
It is equally clear that, for any given $i\in\{i, \dots, l\}$ a hyperplane $H$ such that this identity holds.

Now take $f\in\Diff^1(M)$ to admit a partially hyperbolic $T_x M = E^u \oplus E^{cs}$.
Then there exists $k = k(x) \in\{1, \dots, l-1\}$ such that
$$
E^u(x)=E^1_x\oplus \cdots \oplus E^k_x
\text{ and }
E^{cs}(x)=E^{k+1}_x \oplus \cdots \oplus E^l_x.
$$
We refer to $\lambda_{k+1}(x), \dots, \lambda_l(x)$ as the \emph{center Lyapunov exponents} at $x$.

The partial volume expansion assumption means that $\log\left|\det Df(x)_{\mid H} \right|$ is positive
for any $x\in M$ and any hyperplane $H$ containing the unstable subspace $E_x^u$.
Since the set of such pairs $(x,H)$ is compact (keep in mind that $x \mapsto E_x^u$ is continuous),
we get that there exists $c>0$ such that
\begin{equation}\label{eq.used3}
\log\left|\det Df(x)_{\mid H} \right| \ge c
\text{ for any $x\in M$ and any hyperplane $H\supset E_x^u$.}
\end{equation}
As the set of such pairs $(x,H)$ is also invariant under iteration, it follows that $c$ is a lower
bound for the limit in \eqref{eq.used1}. Thus, our assumption implies that
\begin{equation}\label{eq.used2}
\Delta(x) - \lambda_i(x) \ge c \text{ for $\mu$-almost every $x\in M$ and any $i\in\{k+1, \dots, l\}$}
\end{equation}

\subsection{Gibbs $u$-states}

Following Pesin and Sinai~\cite{PS82} and Bonatti and Viana~\cite{BoV00} (see also~\cite[Chapter 11]{Beyond}),
we call \emph{Gibbs $u$-state} any invariant probability measure whose conditional probabilities
along strong unstable leaves are absolutely continuous with respect to the Lebesgue measure on the leaves.

It is shown in \cite[Section 11.2]{Beyond} (see also Dolgopyat~\cite{Dol04a}) that every physical measure
of a $C^{1+}$ partially hyperbolic diffeomorphism is a Gibbs $u$-state. A partial converse is true:
any ergodic Gibbs $u$-state whose center Lyapunov exponents are all negative is a physical measure.

Let $\Gibb^u(f)$ denote the space of all Gibbs $u$-states. By an \emph{unstable disk} we mean any
embedded disk contained in some unstable leaf of $f$.
By \emph{empirical measures} of a point $x\in M$ we mean any accumulation point of the sequence
${n}^{-1}\sum_{i=0}^{n-1}\delta_{f^i(x)}$.
The proofs of the following basic properties can also be found in~\cite[Section 11.2]{Beyond}:

\begin{proposition}\label{p.Gibbsustates}
Let $f$ be a $C^{1+}$ partially hyperbolic diffeomorphism. Then
\begin{itemize}
\item[(1)] $\Gibb^u(f)$ is non-empty, weak$^*$ compact and convex. Ergodic components of Gibbs $u$-states are Gibbs u-states.
\item[(2)] The support of every Gibbs $u$-state is $\cF^u$-saturated, that is, it consists of entire unstable leaves.
\item[(3)] For Lebesgue almost every point $x$ in any unstable disk, every empirical measure $\nu_x$ is a Gibbs $u$-state.
\item[(4)] For every unstable disk $D^u$, any weak$^*$ limit of the sequence of measures
$$
\frac{1}{n}\sum_{i=1}^n (f^i)_* (\Leb_{D^u})
$$
is a Gibbs $u$-state.
\end{itemize}
\end{proposition}

\begin{remark}\label{r.Gibbsustates}
Since the unstable foliation is absolutely continuous, part (3) of the proposition implies that there is a full volume
subset of points $x\in M$ for which every empirical measure is a Gibbs $u$-state.
\end{remark}

\section{A physical property}\label{physical}

We use the expression \emph{physical property} to refer to any property that holds on a positive
volume measure subset of the ambient manifold for any diffeomorphism.
The physical property is \emph{full} if it holds on a full volume subset.
The main result of this section is the following full physical property for $C^1$ diffeomorphism:

\begin{theorem}\label{t.limit}
Let $U\subseteq M$ be a compact subset of a manifold and $f: U\to \inte(U)$ be an embedding.
Then there is a full volume measure subset $\Gamma$ of $U$ such that
\begin{equation}\label{eq.limit}
\limsup_n \frac 1n \log|Df^n(x)| \le 0 \text{ and } \int_M \log|\det Df| \, d\nu_x \le 0
\end{equation}
for any empirical measure $\nu_x$ of every $x\in \Gamma$.
\end{theorem}

Note that the left hand side of \eqref{eq.limit} is the average, with respect to $\nu_x$,
of the sum of all Lyapunov exponents.
For volume preserving maps, the statement is contained in the Oseledets theorem.

\begin{proof}
Let $\cP$ denote the space of probabilities of $M$ and, for each $r\ge 0$,
$$
\cK_r=\left\{\mu\in\cP: \int_M \log|\det Df| \, d\mu \leq r\right\}.
$$
We are going to show that for each $r>0$ there is a full volume subset $\Gamma_r$ such that,
for every $x\in\Gamma_r$,
$$
\limsup_n \frac 1n \log|Df^n(x)| \le r
$$
and every empirical measure $\nu_x $ of every $x$ belongs to $\cK_r$.
Clearly, we may choose $r\mapsto \Gamma_r$ to be monotone.
Then $\Gamma=\cap_{r>0} \Gamma_r$ is still a full volume subset satisfying all
the conditions in the statement.

Observe that $\cK_r$ is weak$^*$-compact and convex. For fixed $r>0$, denote
$$
\cI_{n,r}=\left\{x\in M: \frac{1}{n}\sum_{i=0}^{n-1}\delta_{f^i(x)}\in \cK_r\right\}.
$$
The definition means that $|\det Df^n(z) \geq e^{nr}$ for every $z\in \cI_{n,r}^c$. Then
$$
\Leb(M) \geq % \Leb(f^n(\cI_{n,r})) =
\int_{\cI{n,r}^c} |\det Df^n| d\Leb(x)
\geq e^{nr} \Leb(\cI_{n,r}^c).
$$
Consequently, $\Leb(\cI_{n,r}^c)\leq \Leb(M) e^{-nr}$ which, by Borel--Cantelli, implies that
$$
\Leb\left(\bigcap_n\bigcup_{k\geq n} \cI_{k,r}^c\right)=0.
$$
Now, the definition also means that if $x$ is in $\bigcup_n\bigcap_{i\geq n} \cI_{i,r}$
then there exists $n\ge 1$ such that
$$
\frac 1n \sum_{i=0}^{k-1}\delta_{f^i(x)} \in \cK_r
\text{ for every } k \ge n.
$$
Then, by compactness, every empirical measure of $\nu_x$ is also in $\cK_r$. This shows that
we may take $\Gamma_r$ to be $\bigcup_n\bigcap_{i\geq n} \cI_{i,r}$.
\end{proof}

The following immediate consequence of the theorem was used previously in \cite{Yan17}:

\begin{remark}
Suppose that $f:M\to M$ is such that $\int \log|\det Df| \, d\mu \geq 0$ for any invariant probability $\mu$.
Then there is a full Lebesgue measure subset $\Gamma$ of $M$ such that
$$
\int \log|\det Df| \, d\nu_x = 0
$$
for any $x\in\Gamma$ and any accumulation point $\nu_x$ of ${n}^{-1}\sum_{i=0}^{n-1} \delta_{f^i(x)}$.
In particular, if there exists a unique invariant probability $\mu$ for which $\int \log|\det Df| \, d\mu = 0$,
and then $\mu$ must be a physical measure (because it must coincide with any $\nu_x$).
\end{remark}

\section{Proof of Theorem~\ref{main.global}}\label{s.Theorem A}

Throughout this section we take $f$ to be a $C^{1+}$ partially volume expanding diffeomorphism.
Let $c>0$ be as in \eqref{eq.used3}.

\begin{lemma}\label{l.simple_fact1}
If $\tilde\mu$ is an ergodic Gibbs $u$-state such that
$$
\int_M \log|\det Df| \, d\tilde\mu \le 0
$$
then $\tilde\mu$ is a physical measure and all its center Lyapunov exponents are less than $-c$.
\end{lemma}

\begin{proof}
Using the Birkhoff ergodic theorem and inequality \eqref{eq.used2},
$$
0 \ge \int_M \log| \det Df| \, d\tilde\mu
= \int_M \Delta \, d\tilde\mu
\ge c + \int_M \hat\lambda \, d\tilde\mu
$$
where $\hat\lambda$ denotes the largest center exponent. Thus, all the center Lyapunov exponents
are less than $-c$. By \cite[Section 11.2]{Beyond}, the fact that the center exponents are negative
ensures that $\tilde\mu$ is a physical measure.
\end{proof}

\begin{corollary}\label{c.physicalmeasure}
There is a full volume measure subset $\Gamma$ of $M$ such that every empirical measure $\nu_x$ of any $x\in\Gamma$
has an ergodic component $\mu_x$ which is a Gibbs $u$-state with $\int \log|\det Df| \, d\mu_x \leq 0$ and whose
center Lyapunov are all smaller than $-c$. In particular, $\mu_x$ is a physical measure.
\end{corollary}

\begin{proof}
By Remark~\ref{r.Gibbsustates} and Theorem~\ref{t.limit}, there is a full volume subset $\Gamma$ of $M$ such that
every empirical measure $\nu_x$ of every $x\in\Gamma$ is a Gibbs $u$-state and satisfies $\int_M \log| \det Df| \, d\nu_x \leq 0$. By part (1) of Proposition~\ref{p.Gibbsustates}, this $\mu_x$ is a Gibbs $u$-state.
Using Lemma~\ref{l.simple_fact1} with $\tilde\mu=\mu_x$, we get the claim of the proposition.
\end{proof}

The ergodic decomposition theorem (see \cite[Chapter~5]{FET}) states given any invariant probability
measure $\mu$, there exists a probability measure $\Phi_\mu$ on the space of $f$-invariant probability
measures on $M$, giving full weight to the subset of ergodic probability measures, and such that
$\mu(E) = \int_\cP \eta(E) \,d\Phi_\mu(\eta)$ for every measurable set $E\subset M$.
The ergodic probability measures in the support of $\Phi_\mu$ are called \emph{ergodic components} of $\mu$.
If $\mu$ is a Gibbs $u$-state, the support of $\Phi_\mu$ is contained in the space of Gibbs $u$-states
(recall Proposition~\ref{p.Gibbsustates}).

\begin{lemma}\label{l.simple_fact2}
If $\tilde\mu$ is an ergodic Gibbs $u$-state whose center Lyapunov exponents are all negative then
it is an isolated point in the set of ergodic Gibbs  $u$-states.
Consequently, $\tilde\mu$ is an ergodic component of some Gibbs $u$-state $\mu$ then
$\Phi_\mu(\{\tilde\mu\})>0$.
\end{lemma}

\begin{proof}
By now, the argument is quite standard (see for instance \cite[Lemma~2.9]{BoV00}).
Let $(\mu_j)_j$ be any sequence of ergodic Gibbs $u$-states accumulating on $\tilde\mu$.
Since $\tilde\mu$ is a hyperbolic measure, almost every point has a (Pesin) local stable manifold.
Since the stable lamination is absolutely continuous, and $\supp\tilde\mu$ is accumulated by the
supports of the $\mu_j$, local stable manifolds of $\tilde\mu$-typical points contain $\mu_j$-typical
points. That implies that $\mu_j=\tilde\mu$ for every large $j$. This proves that $\tilde\mu$ is isolated.
The last claim is a consequence, since $\Phi_\mu$ gives full weight to the set of ergodic Gibbs $u$-states.
\end{proof}

We are ready to prove that under our assumptions $f:M\to M$ has only finitely many physical measures,
and that their basins cover a full volume set. That is done in the next couple of lemmas.

\begin{lemma}\label{l.finite}
The diffeomorphism $f$ has finitely many ergodic Gibbs $u$-states with negative center exponents.
\end{lemma}

\begin{proof}
Suppose that $f$ admits infinitely many distinct ergodic Gibbs $u$-states $\mu_j$ whose center exponents
are all negative.
By compactness of the space of Gibbs $u$-states (part (1) of Proposition~\ref{p.Gibbsustates}),
we may assume that the sequence $\mu_j$ converges in the weak-* topology to some Gibbs $u$-state $\mu$.
As pointed out before, every $\mu_j$ is a physical measure. Thus,
$$
\lim \frac{1}{n}\sum_{i=0}^{n-1}\delta_{f^i(x_n)}=\mu_j
$$
for every $x$ in some positive volume set. Take $x$ in the intersection of such a set with the full volume
set $\Gamma$ in Theorem~\ref{t.limit}, so that
$$
\int \log|\det Df| \, d\mu_j \leq 0.
$$
Making $j \to \infty$, we get that
$$
\int \log|\det Df| \, d\mu \leq 0.
$$
So, by part (1) of Proposition~\ref{p.Gibbsustates}, there is an ergodic component $\tilde{\mu}$ of $\mu$
which is a Gibbs $u$-state of $f$ such that
$$
\int \log|\det Df| \, d\tilde{\mu}\leq 0,
$$
and it is no restriction to assume that the support of $\tilde{\mu}$ is contained in the support of $\mu$.
Thus, using \eqref{eq.used2},
$$
0 \ge \int_M \log| \det Df| \, d\tilde\mu
= \int_M \Delta \, d\tilde\mu
\ge c + \int_M \hat\lambda \, d\tilde\mu
$$
By Lemma~\ref{l.simple_fact1}, it follows that $\tilde\mu$ is a physical measure and its
center Lyapunov exponents are negative.
Now the same argument as in Lemma~\ref{l.simple_fact2} proves that $\mu_j=\tilde\mu$
for every large $j$, a contradiction.
%quite standard (see for instance \cite[Lemma~2.9]{BoV00}).
%Since $\tilde\mu$ is a hyperbolic measure, almost every point has a (Pesin) local
%stable manifold. Since the stable lamination is absolutely continuous,
%and $\supp\tilde\mu$ is accumulated by the supports of the Gibbs $u$-states $\mu_j$,
%local stable manifolds of $\tilde\mu$-typical points contain $\mu_j$-typical points.
\end{proof}

\begin{lemma}\label{l.fullvolume}
The union of the basins of the physical measures $\mu_1,\cdots, \mu_m$ in Lemma~\ref{l.finite}
is a full volume subset of $M$.
\end{lemma}

\begin{proof}
Assume that the complement $\cC$ of the union has positive volume.
Let $\cC_0$ be the intersection of $\cC$ with the full volume set $\Gamma$ in Corollary~\ref{c.physicalmeasure}
and with the full volume set of points  $x\in M$ for which every empirical measure is a Gibbs $u$-state
(recall Remark~\ref{r.Gibbsustates}).
Since the unstable foliation is absolutely continuous, there exists some unstable disk that intersects $\cC_0$
at a positive volume subset $I$. We may take $I$ to be compact.
By part (4) of Proposition~\ref{p.Gibbsustates}, any accumulation point $\mu$ of the sequence
$$
\mu_n=\frac{1}{n}\sum_{j=0}^{n-1}f_*^j(\Leb_{I})
$$
($\Leb_I$ denotes the normalized restriction of the
volume measure to $I$)
is a Gibbs $u$-state. Then, for some subsequence $(n_k)_k$ we have
$$
\begin{aligned}
\int \log|\det Df| \, d\mu
& = \lim_k \int \log|\det Df| \, d\mu_{n_k}
= \lim_k \int_I \frac{1}{n_k} \log|\det Df^{n_k}| \,d\Leb \\
& \le \int_I \limsup_k \frac{1}{n_k} \log|\det Df^{n_k}| \,d\Leb
\le 0,
\end{aligned}
$$
by Theorem~\ref{t.limit}.
Using part (1) of Proposition~\ref{p.Gibbsustates}, it follows that there is an
ergodic component $\tilde\mu$ of $\mu$ which is a Gibbs $u$-state of $f$ and satisfies
$$
\int \log \det(T_xf)\, d\tilde{\mu}(x)\leq 0.
$$
In particular, $\supp\tilde\mu$ is contained in $\supp\mu$.
By Lemma~\ref{l.simple_fact1}, $\tilde\mu$ is a physical measure and its center
Lyapunov exponents are negative.
By Lemma~\ref{l.simple_fact2}, it follows that $a=\Phi_\mu(\{\tilde\mu\})$ is positive.

Let $D$ be an unstable disk such that (volume) almost all points in $D$ are $\tilde\mu$-typical, in the sense of Pesin: in particular, they admit local Pesin stable manifolds.
Then let $B$ be a compact positive volume subset of $D$ (the volume will be denoted as $b$) such that
\begin{itemize}
\item[(i)] the size of the local stable manifolds of points in $B$ is uniformly bounded from below;
\item[(ii)] the holonomy maps of the stable lamination through $B$ are
    uniformly absolutely continuous.
\end{itemize}
The second condition means, more precisely, that there exists $K>1$ such that the
projection $\cH^s_{D_1,D_2}: D_1 \to D_2$ along the stable laminae of $L$
between any two nearby small unstable disks $D_1$ and $D_2$ has a Jacobian that
is bounded above by $K$ and bounded below by $1/K$.

Existence of such a structure follows from
classical Pesin theory, as we observed
in our previous paper~\cite{ViY13}, where we called
$cs$-block the union of the laminae through a set
$B$ with these properties.
Indeed, (i) and (ii) together with the assumption that $L$ consists of $\tilde\mu$-typical points and has positive volume inside $D$, imply that the $cs$-block $\cB$ has positive $\tilde\mu$-measure and positive volume in the ambient manifold, and is contained in the basin of $\tilde\mu$.

We claim that there exists $n\ge 1$ and a positive
volume subset $I_0$ of $I$ such that every point of $f^n(I_0)$ is contained in $\cB$. This yields a contradiction, since $I$ is assumed to be in the complement of the basins of all physical measures.
So, we are left to prove the claim. This is analogous to Proposition~6.9 in~\cite{ViY13},
so we only sketch the arguments.

Let $b>0$ denote the volume of $B$ inside $D$. We say that an unstable disk $L$ crosses the $cs$-block if it
intersects every lamina once. Then $\Leb(L \cap \cB)\ge b/K$. Let $\cU$ be a small open neighborhood of $\cB$.
Obviously, $\Leb(L \cap \cU)\ge b/K$ and $\tilde\mu(\cU)\ge \tilde\mu(\cB)>0$.
Let $c=a\tilde\mu(\cU)$, where $a=\Phi_\mu(\{\tilde\mu\})>0$. Then $\mu(\cU)\ge c >0$.
Next, for each small $\delta>0$ let $J_\delta$ denote the $\delta$-neighborhood of $I_0$ inside the corresponding
unstable leaf which contains $I_0$. Define
$$
\mu^\delta_n=\frac{1}{n}\sum_{j=0}^{n-1}f_*^j(\Leb_{J_\delta})
$$
and let $\mu^\delta$ be an accumulation point along the same subsequence as for $\mu$.
It is clear that $\mu^\delta\to\mu$ as $\delta\to 0$, and so we may assume that $\mu^\delta(\cU)\ge c/2$.
Then, there exist $n$ arbitrarily large such that $f^{-n}(\cU)$ intersects $J_\delta$ on a subset with relative
measure larger than $c/4$.
By considering disjoint disks with bounded diameter inside $f^{n}(J_\delta)$ that cross the $cs$-block,
and using bounded distortion of the volume measure under backwards iterates along such disks, we conclude
that the relative volume of $J_\delta \cap f^{-n}(\cB)$ inside $J_\delta$ is bounded below by some constant $r>0$
that depends only on $c$ and on the distortion bound. In particular, $r$ is independent of $\delta$.
On the other hand,  the relative volume of $I_0$ inside $J_\delta$ goes to $1$ as $\delta$ goes to $0$.
Thus making $\delta\to 0$, we conclude that the relative volume of $I_0 \cap f^{-n}(\cB)$ inside $I_0$
is still bounded below by $r$. This implies the claim, obviously.
\end{proof}

Theorem~\ref{main.global} follows immediately from Lemma~\ref{l.finite} and Lemma~\ref{l.fullvolume}.

\section{Proof of Theorem~\ref{main.relation}}

Recall (from~\cite{BoV00,An10}) that a partially hyperbolic diffeomorphism $f$ has
\emph{mostly contracting center} if the center Lyapunov exponents of every ergodic Gibbs $u$-state
are all negative.

Let $\mu$ be any ergodic Gibbs $u$-state of $f$.
The entropy $h_\mu(f)$ of $\mu$ is larger than or equal to its conditional entropy along the unstable foliation.
Moreover, Ledrappier~\cite{Led84a} implies that for Gibbs $u$-states the latter is equal to
$\int \log|\det Df\mid_{E^u}| \, d\mu$, which is obviously positive. Thus,
\begin{equation}\label{eq.entropyofGibbsu}
h_\mu(f) \geq \int \log|\det Df\mid_{E^u}| d\mu > 0.
\end{equation}
That implies that the smallest center Lyapunov exponent of $\mu$ is strictly negative:
otherwise, all the exponents would be non-negative which, by the Ruelle inequality~\cite{Rue78}
applied to the inverse $f^{-1}$, would imply that $h_\mu(f)=0$.
In particular, this proves the theorem when the $u$-codimension is 1.

In the $u$-codimension 2 case, let $\lambda^c_1(\mu) \geq \lambda^c_2(\mu)$ be the center Lyapunov exponents.
We already know that $\lambda_2^c(\mu)<0$, so are left to proving that the same holds for $\lambda_1^c(\mu)$.
Suppose otherwise: $\lambda_1^c(\mu) \ge 0$.
Since $-\lambda_2^c(\mu)$ is the unique positive exponent of $\mu$ for $f^{-1}$,
the Ruelle inequality gives that $h_\mu(f)\leq -\lambda^c_2(\mu)$.
Combining this with \eqref{eq.entropyofGibbsu}, we get that
\begin{equation}\label{eq.entropyofGibbsu2}
\int \log |Df\mid_{E^u}| d\mu+\lambda^c_2(\mu) \leq 0.
\end{equation}
By the Oseledets theorem (see \cite[Chapter 4]{LLE}) the integral on the left hand side coincides with
the sum of all Lyapunov exponents along the unstable sub-bundle $E^u$.
Thus the left hand side of \eqref{eq.entropyofGibbsu2} is equal to the sum of all Lyapunov exponents
except for $\lambda_1^c(\mu)$, and so the inequality directly contradicts \eqref{eq.used2}.
This contradiction proves the theorem.

\section{Partially volume expanding attracting sets}\label{local}

As a first step, let us state a variation of Theorem~\ref{main.global} for attracting sets of
diffeomorphisms. Then we will check that it applies to the generalized solenoid on the solid
torus $M=S^1\times D$
$$
f_0: M \to M, \quad f_0(\theta,x)=\left(k\theta \mod 1, h_\theta(x)\right)
$$
presented in the Introduction. Since $\|Dh_\theta(x)\|$ and $\|Dh_\theta(x)^{-1}\|$ are taken
to be strictly less than $k$ at every point, we may find $a<1$ such that
\begin{equation}\label{eq_ak}
\frac{1}{ak} \|v\| \le \|Dh_\theta(x) v\| \le a k \|v\|
\text{ for any } v \in T_x D \text{ and } (\theta,x) \in M.
\end{equation}
Let us point out that the dynamics of this embedding can be very complicated, for
instance, it may exhibit infinitely many coexisting periodic repellers (see~\cite{BLY13}).

\subsection{A semi-local finiteness theorem}

Let $f:M\to \inte(M)$ be a $C^{1+}$ embedding of a compact manifold with boundary to its interior.
By the \emph{attracting} set we mean the maximal invariant set
$\Lambda_f = \cap_{n=0}^\infty f^n(M)$.
Observe that the map $f\mapsto\Lambda_f$ is upper semi-continuous with respect to the uniform
topology on $f$ and the Hausdorff topology on the space of compact subsets of $M$.

We call $f$ \emph{partially hyperbolic} if $\Lambda_f$ is a partially hyperbolic set.
A useful equivalent condition (see~\cite{Yoc95}) is that there exist two continuous
families $\cC^{u}$ and $\cC^{cs}$ of closed cones in the tangent space satisfying:
\begin{itemize}
\item[(a)] $\cC^{u}(x)\cap \cC^{cs}(x)=\{0\}$ for any $x\in M$;
\item[(b)] $Df(x)(\cC^u(x))\subset \cC^u(f(x))$ and $Df(x)(\cC^{cs}(x))\supset \cC^{cs}(f(x))$
for any $x\in M$;
\item[(c)] $\|Df(x)u\|>\|u\|$ for any non-zero vector $u \in \cC^u(x)$ and any $x$ close to $\Lambda_f$.
\end{itemize}

Similarly, we call $f$ \emph{partially volume expanding} if $\Lambda_f$ is a partially volume
expanding set, meaning that $\log\left|\det Df(x)_{\mid H} \right|$ is positive for any
$x\in\Lambda_f$ and any hyperplane $H$ of the tangent space $T_x M$ containing the unstable
subspace $E_x^u$. Using the upper semi-continuity of $\Lambda_f$ one can easily that this is a
stable property, that is, it remains true for every $C^1$ perturbation of $f$.
Moreover, the consequences \eqref{eq.used3} and \eqref{eq.used2} still hold in this context.

\begin{theorem}\label{main.local}
Any partially volume expanding $C^{1+}$ embedding $f:M\to\inte(M)$ admits finitely many physical
measures, the union of whose basins is a full volume subset of the ambient manifold.
\end{theorem}

The proof is virtually identical to that of Theorem~\ref{main.global}.
Section~\ref{physical} was already formulated in the language of embeddings, and so it applies
immediately to the present context.
Gibbs $u$-states on the invariant set $\Lambda_f$ are defined in precisely the same way as in the
previous global setting. We denote by $\Gibb^u(f)$ the space of Gibbs $u$-states of $\Lambda_f$.
Proposition~\ref{p.Gibbsustates} remains valid here, as long as we define a $u$-disk to mean any
disk embedded in $\inte(M)$ whose tangent space is contained in the unstable cone at every point.
With that same convention, the arguments in Section~\ref{s.Theorem A} also extend immediately to
this setting. Thus one gets Theorem~\ref{main.local}.

\subsection{A partially volume expanding solenoid}

The final next couple of lemmas assert that generalized solenoid $f_0:M \to M$ is partially
hyperbolic and partially volume expanding. Thus Theorem~\ref{main.local} applies to it and all its
perturbations.

\begin{lemma}\label{l.solenoidpartiallyhyperbolic}
$f_0$ is a partially hyperbolic diffeomorphism.
\end{lemma}

\begin{proof}
Take $a$ as in \eqref{eq_ak} and
$$
K=\max\left\{\left\|\frac{\partial h_\theta}{\partial \theta}(\theta,x)\right\|: (\theta,x)\in M \right\},
$$
and then define
\begin{equation*}
\begin{aligned}
\cC^u(\theta,x)
& = \left\{(u,v)\in T_{(\theta,x)} M: \|v\| \leq \frac{2K}{k(1-a)}\|u\|\right\}\\
\cC^{cs}(\theta,x)
& = \left\{(u,v)\in T_{(\theta,x)} M:  \|v\| \geq \frac{3K}{k(1-a)}\|u\|\right\}.
\end{aligned}
\end{equation*}
It is straightforward to check that these two cone fields satisfy the conditions in the definition above.
\end{proof}

Observe that the two-dimensional vertical sub-bundle $\{0\} \times T_x D$ is invariant under $Df$.
Indeed, it coincides with the center-stable sub-bundle $E^{cs}_{(\theta,x)}$ of $f$.

\begin{lemma}\label{l.solenoidweakcontraction}
There is an integer $N \ge 1$ such that $f_0^N$ is partially volume expanding.
\end{lemma}

\begin{proof}
Let $H$ be any 2-plane that contains the unstable subspace $E^u_{(\theta,x)}$.
Of course, $H$ must intersect the vertical 2-plane $\{0\} \times T_x D$ along some direction $F$.
The iterates of vectors along $E^u_{(\theta,x)}$ remain inside the unstable cone
$\cC^u$ and so their angle to the horizontal direction is uniformly bounded from
$\pi/2$. In particular, their growth rate under iteration is equal to $k$.
For the vectors along $F$, the growth rate is given by the vertical derivative.
In particular, using \eqref{eq_ak}, it is bounded below by $1/(ak)$.
Since the angle between the iterates of $E^u_{(\theta,x)}$ and $F$ are uniformly bounded from zero,
it follows that the rate of growth of the determinant along the 2-plane $H$ is
bounded below by $1/a>1$. Thus, there exists $N \ge 1$ such that
$|\det Df^N(\theta,x)|_{H}$ is strictly
bigger than $1$ at every point of $\Lambda_f$. This proves the claim.
\end{proof}

\end{document}